\documentclass[12pt,a4paper]{amsart}

\usepackage{amsfonts, amsmath, amssymb, amsthm, amscd, hyperref}

\newtheorem{thm}{Theorem}
\newtheorem{lem}{Lemma}

\newtheorem{cor}[thm]{Corollary}

\theoremstyle{definition}
\newtheorem{defn}{Definition}

\theoremstyle{remark}

\renewcommand{\int}{\mathop{\rm int}}

\renewcommand{\epsilon}{\varepsilon}

\begin{document}

\title[KKM-type theorems for products\dots]{KKM-type theorems for products of simplices and cutting sets and measures by straight lines}
\author{R.N.~Karasev}
\address{Roman Karasev, Dept. of Mathematics, Moscow Institute of Physics
and Technology, Institutskiy per. 9, Dolgoprudny, Russia 141700}

\thanks{This research was partially supported by the Dynasty Foundation.}


\begin{abstract}
In this paper a version of Knaster-Kuratowski-Mazurkiewicz theorem for products of simplices is formulated. Some corollaries for measure partition in the plane and cutting families of sets in the plane by lines are given.
\end{abstract}

\maketitle

\section{Introduction}

In this paper we generalize the Knaster-Kuratowski-Mazurkiewicz (KKM) theorem~\cite{kkm1929} to coverings of a product of simplices. This result also generalizes the KKM theorem with several coverings from~\cite{bapat1989} and the version of KKM theorem for a product of two simplices of the same dimension in~\cite{tardos1998}. See also~\cite{berg2005} for similar results, where the a simplex is replaced by the $n$-fold symmetric product of a tree.

Some corollaries are given for partitioning a measure in the plane by lines, parrallel to the coordinate axes. Similar results on cutting families of sets in the plane by lines are presented.

Here we formulate the generalization of the KKM theorem.

\begin{defn}
Denote the $n$-dimensional (geometric) simplex $\Delta^n$. The points of $\Delta^n$ will usually be represented by their barycentric coordinates 
$$
(t_1,\ldots, t_{n+1}),\quad \forall i\ t_i\ge 0,\quad \sum_{i=1}^{n+1} t_i = 1.
$$ 
\end{defn}

\begin{defn}
For a simplex $\Delta^n$ denote its facet for $i=1,\ldots, n+1$
$$
\partial_i \Delta^n = \{(t_1,\ldots, t_{n+1})\in\Delta^n : t_i = 0\}.
$$
A facet of a simplex is an $n-1$-dimensional simplex.
\end{defn}

\begin{defn}
Denote $[n] = \{1, 2, \ldots, n\}$.
\end{defn}

\begin{thm}
\label{tardos} Suppose the product of simplices $\Delta^{n-1}\times\Delta^{m-1}$ $(n\le m)$ is covered with $mn$ open (closed) sets $A_{ij}$ $(i\in [n],j\in [m])$ so that every $A_{ij}$ does not intersect $\partial_i\Delta^{n-1}\times \Delta^{m-1}$ and does not intersect $\Delta^{n-1}\times \partial_j\Delta^{m-1}$. Take some positive integers $a_1,\ldots a_n$ that sum up to $m$ (i.e. $\sum_{i=1}^n a_i = m$).

Then there exists a map $\sigma: [m]\to[n]$ such that
$$
\forall i\in[n]\ |\sigma^{-1}(i)| = a_i,\quad \bigcap_{j=1}^m
A_{\sigma(j)j} \neq \emptyset.
$$
\end{thm}

This theorem generalizes the result from~\cite{tardos1995,tardos1998}, where the case $m=n$ is actually considered. In those papers this theorem was used to deduce a Helly-type theorem on piercing of $2$-intervals, i.e. disjoint unions of two intervals on two separate lines. 

Though Theorem~\ref{tardos} seems to be a slight generalization of known results, it is worth mentioning that it implies most of the results in the author's paper~\cite{kar2005}, because it has the following corollary.

\begin{cor}[Colored KKM theorem]
\label{kkm} Consider $mn$ open (closed) subsets $A_{ij}$ $(i\in [n],j\in [m])$ of a simplex $\Delta^{n-1}$. Suppose that every $A_{ij}$ does not intersect $\partial_i\Delta^{n-1}$, and every family 
$$
\mathcal A_j = \{A_{ij}\}_{i=1}^n
$$
is a covering of $\Delta^{n-1}$. Take some positive integers $a_1,\ldots a_n$ that sum up to $m$ (i.e. $\sum_{i=1}^n a_i = m$).

Then there exists a map $\sigma: [m]\to[n]$ such that
$$
\forall i\in[n]\ |\sigma^{-1}(i)| = a_i,\quad \bigcap_{j=1}^m
A_{\sigma(j)j} \neq \emptyset.
$$
\end{cor}

Next we formulate some corollaries of Theorem~\ref{tardos} for line transversals in the plane.

\begin{cor}
\label{squarepart} Let $\mu$ be an absolutely continuous probabilistic measure in the square $Q = [0,1]\times[0,1]$. We shall consider pairs of partitions of the segments $[0,1] = I_1\cup I_2 \cup\ldots \cup I_n$ and $[0,1] = J_1\cup J_2 \cup\ldots \cup J_m$ ($n\le m$) into lesser segments and the corresponding partitions of $Q$ into rectangles
$$
Q = \bigcup_{i\in[n],j\in[m]} I_i\times J_j.
$$

Fix some $c>0$. Then either there exists a pair of partitions such that for any $i\in[n],j\in[m]$
$$
\mu (I_i\times J_j) < c,
$$
or for any representation of $m$ as a sum of positive integers $m = a_1+a_2+\ldots+a_n$ there exists a pair of partitions $I$ of cardinality $n$, $J$ of cardinality $m$, and a map $\sigma: [m]\to [n]$ such that
$$
\forall i\in[n]\ |\sigma^{-1}(i)| = a_i,\quad \forall j\in[m]\ \mu
(I_{\sigma(j)}\times J_j) \ge c.
$$
\end{cor}

\begin{defn}
Let $\mathcal F$ be a family of subsets of $\mathbb R^2$, and let $\mathcal L$ be a family of lines in $\mathbb R^2$. We say that $\mathcal L$ \emph{cuts} $\mathcal F$ if for any $F\in\mathcal F$ there is some $l\in\mathcal L$ such that $F\cap l\not=\emptyset$.
\end{defn}

\begin{defn}
Fix some coordinate system in $\mathbb R^2$. We call a straight line in $\mathbb R^2$ \emph{horizontal} if it is parallel to the $x$ axis, and we call it \emph{vertical} if it is parallel to the $y$ axis.
\end{defn}

\begin{cor}
\label{linesection}
Consider a finite family $\mathcal F$ of connected bounded open sets in $\mathbb R^2$. Let $n\le m$ be positive integers. Then either $\mathcal F$ can be cut by a family of $m$ horizontal and $n$ vertical lines, or the following statement holds:

For any $n+1$ positive intergers $a_1,\ldots, a_n, a_{n+1}$ that sum up to $m+1$ there exists a subfamily $\{X_1, \ldots, X_{m+1}\}\subseteq \mathcal F$ and a map $\sigma : [m+1]\to[n+1]$ such that
$$
\forall i\in[n+1]\ |\sigma^{-1}(i)| = a_i,
$$ 
the projections of $X_i$ and $X_j$ to the $y$ axis do not intersect if $i\not=j$, the projections of $X_i$ and $X_j$ to the $x$ axis do not intersect if $\sigma(i)\not=\sigma(j)$.
\end{cor}

This corollary is also true for families of compact sets. We can take their open $\varepsilon$-neighborhoods, go to the limit $\varepsilon\to 0$, and use the standard compactness considerations. The next statement is a corollary of Corollary~\ref{linesection}, which is formulated much shorter. It may be considered as a Helly-type theorem for straight line cuts.

\begin{cor}
\label{linesection2}
Consider a finite family $\mathcal F$ of connected bounded open (compact) sets in $\mathbb R^2$. Let $n\le m$ be positive integers. Suppose that every subfamily $\mathcal G\subseteq\mathcal F$ with cardinality $|\mathcal G|\le m+1$ can be cut by either $m$ horizontal, or $n$ vertical lines. Then $\mathcal F$ can be cut by a family of $m$ horizontal and $n$ vertical lines.
\end{cor}

The proof of Theorem~\ref{tardos}, given below, can be generalized to the case of product of more than two simplices. 

\begin{defn}
Denote the full $r$-partite $r$-uniform hypergraph on the vertices $[n]\cup[n]\cup\dots\cup[n]$ ($r$ summands) as $H(n, r)$. The set of edges of this hypergraph is $E(H(n, r)) = [n]^r$.
\end{defn}

\begin{defn}
A \emph{matching} $M$ on a hypergraph $H$ is a subset of its edge set $M\subseteq E(H)$ such that any two distinct members of $M$ are disjoint.
\end{defn}

\begin{thm}
\label{tardos-r} Denote $\Delta_1,\ldots, \Delta_r$ some $r$ simplices of dimension $n$. Suppose the product $\Delta_1\times\Delta_2\times\dots\times\Delta_r$ is covered with $n^r$ open (closed) sets $A_{i_1\dots i_r}$ ($j\in[r]$, $i_j\in [n]$) so that every $A_{i_1\dots i_r}$ does not intersect any of the sets (for any $l\in [r]$)
$$
\Delta_1\times\dots\times\partial_{i_l} \Delta_l\times\dots\times\Delta_r.
$$
Then there is a matching $M$ on $H(n,r)$ of size at least $\lceil\dfrac{n}{r-1}\rceil$ such that the family
$\{A_{i_1\dots i_r} : (i_1,\ldots, i_r)\in M\}$ has a common point.
\end{thm}

This theorem gives more straightforward (from the author's viewpoint) way to prove Theorem~1.4 from~\cite{kaiser1997} on piercing of $r$-intervals. Another approach to $r$-intervals, giving the same result, can be found in~\cite{berg2005}.

\section{Proof of Theorem~\ref{tardos}}

First we need some lemmas.

\begin{lem}
\label{simmap} Let a continuous map $f:\Delta^d\to\Delta^d$ map every face of $\Delta^d$ to itself. Then the cohomology map $f^*: H^d(\Delta^d, \partial\Delta^d)\to H^d(\Delta^d, \partial\Delta^d)$ is the identity map and the map $f$ is surjective.
\end{lem}

\begin{proof}
Let us use induction on $d$. The case $d=0$ is trivial. Consider a facet $F\subset\Delta^d$. This is a simplex of dimension $d-1$ and the statement is true for $F$.

Put $K = \partial\Delta^d\setminus\int F$. The map 
$$
f|_{F}^*:  H^{d-1} (F, \partial F)\to H^{d-1} (F, \partial F)
$$
is the identity map by the inductive assumption.

By the excision isomorphism the cohomology map is the identity on $H^{d-1}(\partial\Delta^d, K)$ and, therefore, on $H^{d-1}(\partial\Delta^d)$. Now the natural coboundary isomorphism $\partial: H^{d-1}(\partial\Delta^d)\to H^d(\Delta^d, \partial\Delta^d)$ along with the commutative diagram
$$
\begin{CD}
H^{d-1}(\partial\Delta^d) @<{f^*}<< H^{d-1} (\partial\Delta^d)\\
@V{\partial}VV @V{\partial}VV\\
H^d(\Delta^d, \partial\Delta^d) @<{f^*}<< H^d (\Delta^d, \partial\Delta^d)
\end{CD}
$$
tells that $f^*$ is the identity of $H^d(\Delta^d, \partial\Delta^d)$.

Now let us prove the surjectivity. Assume the contrary: let $x\in\Delta^d$, and let $x\not\in f(\Delta^d)$. By the inductive assumption $x\not\in\partial\Delta^d$. So the map $f$ passes through the inclusion $i : (\Delta^d\setminus\{x\}, \partial\Delta^d)\to (\Delta^d, \partial\Delta^d)$, but the cohomology $H^d(\Delta^d\setminus\{x\}, \partial\Delta^d)=H^d(\partial\Delta^d, \partial\Delta^d)=0$. This is a contradiction.
\end{proof}

The following lemma is a generalized Hall's lemma, as it was used in~\cite{kar2005}.

\begin{lem}
\label{hall}
Let $G$ be a bipartite graph on vertices $V\cup W$, $|V| = n$, $|W|=m$ $(n\le m)$. Let $a(v)$ for every $v\in V$ be a positive integer, and let $\sum_{v\in V} a(v) = m$. 

Suppose that for every nonempty subset $V'\subseteq V$ the number of vertices in $W$ that are adjacent to some vertex $V'$ is at least $\sum_{v\in V'} a(v)$. 

Then there exists a map $\sigma: W\mapsto V$ such that $\forall w\in W$ the pair $(w, \sigma(w))$ is an edge of $G$ and $\forall v\in V$ $|\sigma^{-1}(v)| = a(v)$.
\end{lem}

Now we are ready to prove Theorem~\ref{tardos}. By the standard neighborhood-compactness reasoning we only have to consider the case of open $A_{ij}$. Then we pass to a partition of unity $\phi_{ij} : \Delta^{n-1}\times\Delta^{m-1}\to\mathbb R$, subordinated to the covering $A_{ij}$.

Consider the functions
$$
f_i(x) = \sum_{j=1}^m \phi_{ij}(x)\quad g_j(x) = \sum_{i=1}^n
\phi_{ij}(x).
$$
These two sets of functions give the maps 
$$
f: \Delta^{n-1}\times\Delta^{m-1}\to\Delta^{n-1},\quad g: \Delta^{n-1}\times\Delta^{m-1}\to\Delta^{m-1}
$$
Now fix some inclusions $\iota_1 : \Delta^{n-1}\times\{y\}\to \Delta^{n-1}\times\Delta^{m-1}$ and $\iota_2:\{x\}\times\Delta^{m-1}\to \Delta^{n-1}\times\Delta^{m-1}$. It is easy to see that the conditions of Lemma~\ref{simmap} are satisfied for $f\circ\iota_1$ and $g\circ\iota_2$. Indeed, $f_i$ is zero on $\partial_i\Delta^{n-1}\times \Delta^{m-1}$ and $g_j$ is zero on $\Delta^{n-1}\times \partial_j\Delta^{m-1}$. Then by Lemma~\ref{simmap} the maps
$$
f^* : H^{n-1}(\Delta^{n-1}, \partial \Delta^{n-1})\to
H^*(\Delta^{n-1}\times\Delta^{m-1},\partial\Delta^{n-1}\times\Delta^{m-1})
$$
and
$$
g^* : H^{m-1}(\Delta^{m-1}, \partial \Delta^{m-1})\to
H^*(\Delta^{n-1}\times\Delta^{m-1},\Delta^{n-1}\times\partial\Delta^{m-1})
$$
are identity maps. From the standard properties of $\times$-product and $\cup$-product on cohomology the map $(f\times g)^*$ is the identity map on $H^{n+m-2}(\Delta^{n-1}\times\Delta^{m-1}, \partial(\Delta^{n-1}\times\Delta^{m-1}))$.

Similar to the proof of Lemma~\ref{simmap}, the map $f\times g$ is surjective and there exists some $x\in\Delta^{n-1}\times\Delta^{m-1}$ such that
$$
\forall i\in[n]\ f_i(x) = \frac{a_i}{m}\quad \forall j\in[m]\
g_j(x) = \frac{1}{m}.
$$
Considering the matrix $\phi_{ij}(x)$ we see that its $i$-th row sum is $a_i/m$ and its column sums all equal $1/m$. Consider the bipartite graph $G$ on vertices $[n]\times[m]$ such that $(i, j)\in G$ iff $\phi_{ij}(x)>0$. It is easy to see that the conditions of Lemma~\ref{hall} are satisfied for $G$. Hence there exists a map $\sigma: [m]\to[n]$ such that
$$
\forall i\in[n]\ |\sigma^{-1}(i)| = a_i,\quad\forall j\in[m]\ \phi_{\sigma(j)j}(x)> 0,
$$
and it follows that $x\in A_{\sigma(j)j}$ for any $j\in[m]$.

\section{Proof of Theorem~\ref{tardos-r}}

Instead of Hall's lemma we use the following lemma that follows from Theorem~1.1 in~\cite{fks1993}.

\begin{lem}
\label{hmatching}
Let the hypergraph $H$ be a subhypergraph of $H(n, r)$. Let edges of $H$ have weights $w(e)$ so that for any vertex $v\in V(H)$
$$
\sum_{e\ni v} w(e) = 1/n
$$
and 
$$
\sum_{e\in E(H)} w(e) = 1.
$$
Then $H$ has a matching of size $\lceil \dfrac{n}{r-1}\rceil$.
\end{lem}

Then, as in the proof of Theorem~\ref{tardos} we pass to the partition of unity by functions $\phi_{i_1\dots i_r}$. We may consider the indexes $(i_1, \ldots, i_r)$ as edges of $H(n, r)$.

Then we consider the functions on $\Delta_1\times\dots\times\Delta_r$ for $k\in [r], l\in [n]$
$$
f_{kl}(x) = \sum_{i_k = l} \phi_{i_1\dots i_r} (x).
$$
For any fixed $k$ the functions $f_{kl}$ give a map $g_k : \Delta_1\times\dots\times \Delta_r\to \Delta^{n-1}$ that induces an identity map
$$
g_k^* : H^{n-1}(\Delta^{n-1}, \partial\Delta^{n-1}) \to H^{n-1} (\Delta_1\times\dots\times \Delta_r, 
\Delta_1\times\dots\times\partial\Delta_k\times\dots\times\Delta_r)
$$
by Lemma~\ref{simmap}.

Again, the Cartesian product of maps $g=g_1\times\dots\times g_r$ gives an identity map
$$
g^* : H^{n-1}((\Delta^{n-1})^r, \partial(\Delta^{n-1})^r) \to H^{n-1} (\Delta_1\times\dots\times \Delta_r, \partial(\Delta_1\times\dots\times\Delta_r)),
$$
and the map $g$ is surjective. Then from the surjectivity we find $x\in\Delta_1\times\dots\times\Delta_r$ such that for any $k\in[r], l\in [n]$
$$
f_{kl}(x) = \sum_{i_k=l} \phi_{i_1\dots i_r}(x) = 1/n.
$$
Denote $H$ the hypergraph with the set of edges 
$$
E(H) = \{(i_1,\ldots, i_r)\in [n]^r : \phi_{i_1\dots i_r}(x) > 0\},
$$
and let the weight of an edge $(i_1,\ldots, i_r)$ be $\phi_{i_1\dots i_r}(x)$. By Lemma~\ref{hmatching} there is a matching on $H$ of size at least $\lceil\dfrac{n}{r-1}\rceil$ and the proof is complete.

\section{Proofs of the corollaries}

\begin{proof}[Proof of Corollary~\ref{squarepart}]
The space of partition pairs of cardinalities $n$ and $m$ is parameterized by $P = \Delta^{n-1}\times\Delta^{m-1}$, if we take the lengths of the segments of the partition as barycentric coordinates in a simplex. Denote the subsets of $P$
$$
A_{ij} = \{(I_1,\ldots,I_n,J_1,\ldots,J_m) : \mu(I_i\times J_j) \ge c\}.
$$
The sets $A_{ij}$ satisfy the condition on the intersection with boundary of Theorem~\ref{tardos}. If they cover $P$ then by Theorem~\ref{tardos} for any number partition $a_1+\ldots+a_n=m$ there is a map $\sigma: [m]\to[n]$ and a point $x\in P$ such that for every $j\in[m]$ $x\in A_{\sigma(j)j}$. This is equivalent to the second alternative of the corollary.

If the sets $A_{ij}$ do not cover $P$, then the first alternative holds.
\end{proof}

\begin{proof}[Proof of Corollary~\ref{linesection}]
We can assume that all the sets $\mathcal F$ are contained in some square $Q=[a, b]\times[a, b]$. Moreover, we can apply a homothety and consider the case $Q=[0, 1]\times[0, 1]$.

The configuration space of all families of $n$ horizontal and $m$ vertical lines, intersecting $Q$ is naturally isomorphic to the space of pairs of partitions of segments with cardinalities $n+1$ and $m+1$. Denote the configuration space $P = \Delta^n\times\Delta^m$ and define its subspaces
$$
A_{ij} = \{(I_1,\ldots,I_{n+1},J_1,\ldots,J_{m+1}) : \exists X\in\mathcal F\ :\ X\subset \int(I_i\times J_j)\}.
$$
The subsets $A_{ij}\subseteq P$ are closed and satisfy the condition on the intersection with boundary of Theorem~\ref{tardos}. 

If the sets $A_{ij}$ cover $P$, then for any number partition $a_1+\ldots+a_{n+1}=m+1$ there is a corresponding map $\sigma: [m+1]\to[n+1]$ and the second alternative of this corollary holds.

If $A_{ij}$ do not cover $P$, then the first alternative of this corollary holds.
\end{proof}

\end{document}